\documentclass[11pt, a4paper]{amsart}
\usepackage{amsmath,textcomp, graphicx}
\usepackage{color}
\usepackage[T1]{fontenc}
\usepackage{amsfonts}
\usepackage{amsmath, amsthm, amssymb}
\usepackage[a4paper,margin=2cm,vmargin={1cm,2cm},includeheadfoot]{geometry}
\usepackage{enumerate}
\usepackage{epstopdf}
\usepackage{subfigure}
\usepackage[scaled=0.92]{helvet}
\usepackage{srcltx}
\usepackage{mathrsfs}

\newtheorem{theorem}{Theorem}[section]
\newtheorem{lemma}[theorem]{Lemma}

\theoremstyle{definition}
\newtheorem{definition}[theorem]{Definition}

\theoremstyle{remark}
\newtheorem{remark}[theorem]{Remark}
\theoremstyle{corollary}
\newtheorem{corollary}[theorem]{Corollary}
\theoremstyle{proposition}
\newtheorem{proposition}[theorem]{Proposition}

\newcommand{\N}{\mathbb{N}}
\newcommand{\M}{\mathbb{M}}

\newcommand{\D}{\mathbb{D}}
\newcommand{\R}{\mathbb{R}}
\newcommand{\E}{\mathbb{E}}
\newcommand{\F}{\mathcal{F}}
\newcommand{\DEF}{\mathrel{\mathrel{\mathop:}=}}

\begin{document}

\title[Functional limit theorem  for integrals driven by some L\'{e}vy processes]{A functional limit theorem for stochastic integrals driven by a time-changed symmetric $\alpha$-stable L\'{e}vy process \thanks{Supported
    by the Institute of Mathematical Statistics (IMS) and the Bernoulli
    Society.}} 
    
\date{\today}

\author[E. Scalas]{Enrico~Scalas}
\address{DISiT - Universit\`a del Piemonte Orientale, viale Michel 11, I15121, Italy and BCAM - Basque Center for Applied Mathematics, Mazarredo 14, E48009 
Bilbao, Basque Country - Spain.}
\email  {enrico.scalas@unipmn.it}

\author[N. Viles]{No\`{e}lia~Viles}
\address{BCAM - Basque Center for Applied Mathematics, Mazarredo 14, E48009 
Bilbao, Basque Country - Spain.} 
\email{nviles@bcamath.org}

\begin{abstract}
Under proper scaling and distributional assumptions, we prove the convergence in the Skorokhod space endowed with the $M_1$-topology of 
a sequence of stochastic integrals of a deterministic function driven by a time-changed symmetric $\alpha$-stable L\'{e}vy process. The time
change is given by the inverse $\beta$-stable subordinator.
\end{abstract}



\maketitle                                      


\section{Introduction}\label{introduction}

\noindent The main motivation of the paper comes from the physical model studied in \cite{S11}. That paper concerns a damped harmonic oscillator subject to a random force. The random force is usually modeled by Gaussian white noise assuming a large number of independent interactions bounded in time and in their strength. However, in many situations of practical interest, the force is a L\'evy process with heavy-tails in its distribution (for instance, an $\alpha$-stable L\'{e}vy process).

In this basic model, the equation of motion is informally given by
\begin{equation}\label{stringeq}
\ddot{x}(t)+\gamma\dot{x}(t)+kx(t)=\xi(t),
\end{equation}
where $x(t)$ is the position of the oscillating particle with unit mass at time $t$, $\gamma>0$ is the damping  coefficient, $k>0$ is the spring constant and $\xi(t)$ represents white L\'evy noise.
This noise term is a formal derivative of a symmetric $\alpha$-stable L\'evy process $L_{\alpha}(t)$. 

The formal solution of Equation (\ref{stringeq}) is
\begin{equation}\label{stringeqsolx}
x(t)=F(t)+\int_{-\infty}^{t}G(t-t')\xi(t')dt',
\end{equation}
where $G(t)$ is the Green function for the corresponding homogeneous equation. It is given by
\begin{equation}\label{Gt}
G(t)=
\left\{
\begin{array}{cc}
  \frac{\exp(-\gamma t/2)}{\sqrt{\omega^2-\gamma^2/4}}\sin(\sqrt{\omega^2-\gamma^2/4}t)  & \omega=\sqrt{k}>\gamma/2,\\
   t\exp(-\gamma t/2)  &  \omega=\gamma/2 ,\\
    \frac{\exp(-\gamma t/2)}{\sqrt{\gamma^2/4-\omega^2}}\sinh(\sqrt{\gamma^2/4-\omega^2}t)  &   \omega<\gamma/2,
\end{array}
\right.
\end{equation} 
and $F(t)$ is a decaying function (a solution of the homogeneous equation under given initial conditions). The solution for the velocity component can be written as 
\begin{equation}\label{stringeqsolv}
v(t)=F_v(t)+\int_{-\infty}^{t}G_v(t-t')\xi(t')dt',
\end{equation}
where $F_v(t)=\frac{d}{dt}F(t)$ and $G_v(t)=\frac{d}{dt}G(t)$ is the Green function for the equation satisfied by the velocity $v(t)$.

In stationary situations the functions $F(t)$ and $F_v(t)$ vanish. Notice that the Green function for Equation (\ref{stringeq}) is bounded for all $t\geqslant 0$. It can be obtained using Laplace transform methods.

Our idea is to replace the white noise in Equation \eqref{stringeq} first with a sequence of instantaneous shots of random amplitude at
random times and then with an appropriate functional limit of this process. We can express the sequence of random instantaneous shots
at random times in terms of the formal derivative of a compound renewal process,  a random walk subordinated
to a counting process called continuous-time random walk (CTRW) by physicists.
As an example, consider the sequence $\{Y_i\}_{i\in\N}$ of i.i.d. symmetric $\alpha$-stable random variables that represent the force jumps, with $\alpha\in (0,2]$.  Let  $\{J_i\}_{i\in \N}$ denote the waiting times between two jumps.  Assume that $\{J_i\}_{i\in \N}$ are i.i.d. and positive such that $J_1$ belongs to the strict domain of attraction of some stable random variables with index $\beta\in (0,1]$. The sum $T_n=\sum_{i=1}^n J_i$ represents the time or {\em epoch} of the $n$-th jump and 
\begin{equation}
\label{counting}
N_{\beta}(t)=\max\{n:T_n\leqslant t\}
\end{equation}
is the number of jumps up to time $t>0$. Then, the compound renewal process is given by
\begin{equation}
\label{ctrw}
X(t) = \sum_{i=1}^{N_\beta (t)} Y_i = \sum_{i \geqslant 1} Y_i \textbf{1}_{\{T_i \leqslant t\}},
\end{equation}
whose paths we take right continuous and with left limits; the corresponding white noise can be formally written as
\begin{equation}
\label{ctrwwn}
\Xi (t) = d X(t)/dt =  \sum_{i=1}^{N_\beta (t)} Y_i \delta(t - T_i) = \sum_{i \geqslant 1} Y_i \delta(t-T_i) \textbf{1}_{\{T_i \leqslant t\}}.
\end{equation}

A natural way to study the convergence of stochastic processes whose paths are right continuous with left limits (c\`{a}dl\`{a}g in french) is by using the Skorokhod topology. The Skorokhod space provides a suitable formalism for the description of stochastic processes with jumps, such as Poisson processes, L\'{e}vy processes, martingales and semimartingales. Since the processes which we will study are semimartingales, we will consider the Skorokhod space endowed with the convenient topology. This topology allows us to move the jumps of the approximants to the times of the jumps of the limit process by using changes of time. 

Our goal is to study the convergence of the integral of a deterministic continuous and bounded function with respect to a properly rescaled 
continuous time random walk. In particular, we aim to prove that under a proper scaling and distributional assumptions, it follows that 
\begin{equation}\label{motivationlimit}
\left\{\sum_{i=1}^{N_{\beta}(nt)}G_v\left(t-\frac{T_i}{n}\right)\frac{Y_i}{n^{\beta/\alpha}}\right\}_{t\geqslant 0}\overset{M_1-top}{\Rightarrow}\left\{\int_0^{t}G_v(t-s)dL_{\alpha}(D^{-1}_{\beta}(s))\right\}_{t\geqslant 0},
\end{equation}
and
\begin{equation}\label{motivationlimitbis}
\left\{\sum_{i=1}^{N_{\beta}(nt)}G \left(t-\frac{T_i}{n}\right)\frac{Y_i}{n^{\beta/\alpha}}\right\}_{t\geqslant 0}\overset{M_1-top}{\Rightarrow}\left\{\int_0^{t}G (t-s)dL_{\alpha}(D^{-1}_{\beta}(s))\right\}_{t\geqslant 0},
\end{equation}
when $n\to+\infty$, in the Skorokhod space $D([0,+\infty),\R)$ endowed with the $M_1$-topology, where $L_{\alpha}(D^{-1}_{\beta}(t))$ is a time-changed L\'{e}vy process $L_{\alpha}$ with respect to the functional inverse $D_{\beta}^{-1}(t)$ of a $\beta$-stable subordinator denoted by $D_{\beta}(t)$ and that we will define in the next sections.

The paper is structured as follows. In Section 2, we define the continuous time random walk and also the compound fractional Poisson process. Since we are interested in the convergence of stochastic processes in Skorokhod spaces we need to introduce the definition of these spaces and the notion of convergence in the appropriate topologies. This is done in Section 3, where we introduce the Skorokhod space with the two most useful topologies on it, called $J_1$ and $M_1$ topologies. In Section 4 we state and prove a functional limit theorem for the case in which the jumps are i.i.d. strictly stable random variables of index $\alpha\in(1,2]$ and waiting times are in the domain of attraction of stable subordinator of index $\beta\in(0,1)$.  We also extend the functional limit theorem to general strictly stable processes of index $\alpha\in(0,2]$. Finally, in Section 5, we discuss some possible extensions of our result. 

\section{Continuous time random walks}
The continuous time random walk (CTRW) was introduced by Montroll and Weiss \cite{MW} in 1965. As its name suggests, the CTRW generalizes the discrete random walk  imposing a random waiting time between particle jumps.  The random walk is a stochastic process given by a sum of independent and identically distributed (i.i.d.) random variables.  A CTRW is a pure jump process given by a sum of i.i.d. random jumps  $(Y_i)_{i\in \N}$ separated by i.i.d.  random waiting times (positive random variables) $(J_i)_{i\in \N}$. Throughout this paper, we shall further assume that for a given $i$, $Y_i$ and $J_i$ are independent random variables.
CTRWs became a widely used tool for describing random process that appear in a large variety of physical models and in finance 
(\cite{M00, M03, Masol01, Masol3, MW, R02, Sab02, S00}). In physics, they are often used to model anomalous diffusion and their one-point distribution 
may exhibit heavy tails. 

As mentioned in the Introduction, CTRWs can be formally defined as random walks subordinated to a counting renewal process. To do this, it is necessary to introduce the random walk  and the counting process. Note that, in defining these processes, it is not necessary to refer to their
meaning in Equations \eqref{stringeq} and \eqref{ctrwwn}, but we can use the interpretation in terms of a diffusing particle on the real line,
in order to help intuition.

Let $X_n=\sum_{i=1}^n Y_i$ denote the position of a diffusing particle after $n$ jumps and $T_n=\sum_{i=1}^nJ_i$ be the epoch of the $n$-th jump. The corresponding counting process $N(t)$ is defined by 
\begin{equation}
\label{generalcounting}
N(t)\overset{\textrm{def}}{=}\max\{n: \; T_n \leqslant t\}.
\end{equation}
Then the position of a particle at time $t>0$ can be expressed as the sum of the jumps up to time $t$
\begin{equation}
\label{generalctrw}
X(t)=X_{N(t)}\overset{\textrm{def}}{=}\sum_{i=1}^{N(t)} Y_i.
\end{equation}
From this expression, we can see that the limiting behavior of the CTRW depends on the distribution of its jumps and its waiting times. 

\subsection{Compound Fractional Poisson Process}
In the Introduction, we mentioned a white-noise force formally defined as the first derivative of special CTRWs. An important
instance of such CTRWs is the compound fractional Poisson process.
Consider a CTRW whose i.i.d. jumps $\{Y_i\}_{i\in \N}$ have symmetric $\alpha$-stable distribution with $\alpha \in (0,2]$, and whose i.i.d waiting times $\{J_i\}_{i\in \N}$ satisfy
\begin{equation}\label{pbeta0t}
\mathbb{P}(J_i>t)=E_{\beta}(-t^{\beta}),
\end{equation} 
for $\beta\in (0,1]$, where 
\[E_{\beta}(z)=\sum_{j=0}^{+\infty}\frac{z^j}{\Gamma(1+\beta j)},\]
denotes the Mittag-Leffler function. If $\beta=1$, the waiting times are exponentially distributed with parameter $\mu =1$ and the counting
process \eqref{generalcounting} is the Poisson process. 
The counting process associated to the renewal process defined by (\ref{pbeta0t}) is called the fractional Poisson process which is a generalization of the Poisson process
\[N_{\beta}(t)=\max\{n : \; T_n\leqslant t\}.\]
It was introduced and studied by Repin and Saichev \cite{RS00}, Jumarie \cite{Jumarie01}, Laskin \cite{Laskin03}, Mainardi et al. \cite{MGS04,MGV07}, Uchaikin et al. \cite{UCS08} and Beghin and Orsingher \cite{BO09, BO10}. 
For $\beta\in (0,1)$, the fractional Poisson process is semi-Markov. However, the process $N_{\beta}(t)$ is not Markov as the waiting times $J_i$ are not distributed exponentially. 
It is not L\'{e}vy because its distribution is not infinitely
divisible\[ \mathbb{P}(N_{\beta}(t)=n)=\frac{t^{\beta n}}{n!}E_{\beta}^{(n)}(-t^{\beta}),\]
where $E_{\beta}^{(n)}(-t^{\beta})$ denotes the $n$-th derivative of $E_{\beta}(z)$ evaluated at the point $z=-t^{\beta}$.

If we subordinate a random walk to the fractional Poisson process, we obtain the compound fractional Poisson process, which is not Markov
\begin{equation}
\label{cfpp}
X_{N_{\beta}(t)}=\sum_{i=1}^{N_{\beta}(t)} Y_i.
\end{equation}
As already mentioned, we are interested in the limiting behaviour of stochastic integrals driven by this process \cite{GPSS09}.

\section{Convergence in the Skorokhod space}\label{Skortops}

The Skorokhod space, denoted by $\D = D([0, T] , \R)$  (with $T>0$), is the space of real functions $x : [0, T] \to \R$
that are right-continuous and have left limits:
\begin{enumerate}
\item For $t \in [0, T)$, $x(t+)=\lim_{s\downarrow t}x(s)$ exists and $x(t+)=x(t)$. 
\item For $t \in (0, T]$, $x(t-)=\lim_{s\uparrow t}x(s)$ exists. 
\end{enumerate}
Functions satisfying these properties are called c\`adl\`ag (a French acronym for \textsl{continu \`{a} droite, limite \`{a} gauche}) functions. 

There is a metrizable topology on $\D$, called the Skorokhod topology, for which this space is Polish and for which convergence of sequences can be characterized as follows.  
Actually, Skorokhod \cite{Skorokhod} proposed four metric separable topologies on $\D$, denoted by $J_1$, $J_2$, $M_1$ and $M_2$. The definition of these topologies is not very intuitive at a first sight. The compact sets of the Skorokhod space are characterized by applying Arzel\`{a}-Ascoli theorem. 

In the space of continuous functions, two functions $x$ and $y$ are close to each other in the uniform topology if the graph of $x(t)$ can be carried onto the graph of $y(t)$ by uniformly small perturbation of the ordinates, with the abscissae kept fixed. In $\D$, we shall allow also a uniformly small deformation of the time scale. The topologies devised by Skorokhod embody this idea.

Let $\Lambda$ denote the class of strictly increasing, continuous mappings of $[0,T]$ onto itself. If 
$\lambda\in \Lambda$, then $\lambda(0)=0$ and $\lambda(T)=T$. For $x,y\in \D$, let us define the distance $d_{J_1}(x,y)$ to be the infimum of those positive $\varepsilon$ for which there exists in $\Lambda$ a $\lambda$ such that
\begin{equation}\label{dj1eq1}
\sup_{t\in[0,T]}|\lambda(t)-t|\leqslant \varepsilon,
\end{equation}
 and
 \begin{equation}\label{dj1eq2}
\sup_{t\in[0,T]}|x(t)-y(\lambda(t))|\leqslant \varepsilon.
\end{equation}
That is, 
\begin{equation}\label{dj1}
d_{J_1}(x,y):=\inf_{\lambda\in \Lambda}\{\sup_{t\in[0,T]}|\lambda(t)-t|,\sup_{t\in[0,T]}|x(t)-y(\lambda(t))|\}
\end{equation} 
This metric is defined in the book by Billingsley (p.111, \cite{billingsley68}). The induced topology is called Skorokhod's $J_1$-topology.
It is now possible to define convergence in this topology.

\begin{definition}[$J_1$-topology] 
The sequence $x_n(t)\in \D$ converges to $x_0(t)\in \D$ in the $J_1-$topology if there exists a sequence of increasing homeomorphisms $\lambda_n:[0,T]\to[0,T]$ such that  
\begin{equation}\label{convJ1}
\sup_{t\in[0,T]}|\lambda_n(t)-t|\to 0, \sup_{t\in[0,T]}|x_n(\lambda_n(t))-x_0(t)|\to 0, 
\end{equation}
as $n\to \infty$.
\end{definition}
Among Skorokhod's topologies, the $J_1$-topology is the finest and the closest one to the uniform topology. If $x_0$ is continuous, the convergence in $J_1$-topology is equivalent to the uniform convergence. The space $\D$ endowed  with $J_1$-topology is complete,  as Kolmogorov showed in  \cite{Kolmogorov56}. 

The $M_1$-topology on $\D$ is generated by the metric $d_{M_1}$ defined by means of completed graphs. For $x\in \D$, the completed graph of $x$ is the set
\[\Gamma_x^{(a)}=\{(t,z)\in[0,T]\times\R:z=a x(t-)+(1-a)x(t)\; \textrm{for}\; \textrm{some} \; a\in[0,1]\},\]
where $x(t-)$ is the left limit of $x$ at $t$ and $x(0-):=x(0)$. A parametric representation of the completed graph $\Gamma_x^{(a)}$ is a continuous nondecreasing function $(r,u)$ mapping $[0,1]$ onto $\Gamma_x^{(a)}$, with $r$ being the time component and $u$ being the spatial component. Let $\Pi(x)$ denote the set of parametric representations of the graph $\Gamma_x^{(a)}$. For $x_1,\; x_2\in\D$ we define the $M_1$-metric on $\D$ as
\begin{equation}\label{dm1}
d_{M_1}(x_1,x_2):=\inf_{\substack{(r_i,u_i)\in \Pi(x_i)\\
i=1,2}}\{ \|r_1-r_2\|_{[0,T]}\vee \|u_1-u_2\|_{[0,T]} \},
\end{equation} 
where $w_1\vee w_2:=\max\{w_1,w_2\}$ for $w_1,\,w_2\in \R$, and $\|u\|_{[0,T]}:=\sup_{t\in[0,T]}\{|u(t)|\}$.

Notice that $\|r_1-r_2\|_{[0,T]}\vee ||u_1-u_2||_{[0,T]} $ can also be written as $\|(r_1,u_1)-(r_2,u_2)\|_{[0,T]}$ where
\begin{align}\label{normm1}
\|(r_1,u_1)-(r_2,u_2)\|_{[0,T]}&:=\sup_{t\in [0,T]}\{|(r_1(t),u_1(t))-(r_2(t),u_2(t))|\}
\nonumber\\&=\sup_{t\in [0,T]}\{|r_1(t)-r_2(t)|\vee |u_1(t)-u_2(t))|\}.
\end{align}
Then we have the following equivalent expression for the $M_1$-metric
\begin{equation}\label{dm11}
d_{M_1}(x_1,x_2)=\inf_{\substack{(r_i,u_i)\in \Pi(x_i)\\
i=1,2}}\left\{ \sup_{t\in [0,T]}\{|r_1(t)-r_2(t)|\vee |u_1(t)-u_2(t))|\} \right\}.
\end{equation}
This metric induces a topology, called Skorokhod's $M_1$-topology which is weaker than the $J_1$-topology.
One of the advantages of the $M_1$-topology is that it allows for a jump in the limit function $x\in\D$ to be approached by multiple jumps in the converging functions $x_n\in\D$. The convergence in this topology is defined as follows.

\begin{definition}[$M_1$-topology] The sequence $x_n(t)\in \D$ converges to  $x_0(t)\in\D$ in the $M_1$-topology if
\begin{equation}\label{convM11}
\lim_{n\to+\infty}d_{M_1}(x_n(t),x_0(t))=0.
\end{equation} 
In other words, we have the convergence in $M_1$-topology if there exist parametric representations $(y(s), t(s))$ of the graph $\Gamma_{x_0(t)}$ and $(y_n(s), t_n(s))$ of the graph $\Gamma_{x_n(t)}$ such that
\begin{equation}\label{convM1}
\lim_{n\to\infty}\|(y_n,t_n)-(y,t)\|_{[0,T]}=0.
\end{equation}
\end{definition}

In our problem, it will be convenient to consider the function space $D([0,+\infty),\R)$ with domain $[0,+\infty)$ instead of the compact domain $[0,T]$. In that setting, we can equivalently define convergence of $x_n\in D([0,+\infty),\R)$ to $x_0\in D([0,+\infty),\R)$  as $n\to\infty$ with some topology ($J_1$-topology or $M_1$-topology) to be convergence $x_n\in D([0,T],\R)$ to $x_0\in D([0,T],\R)$  as $n\to\infty$ with that topology for the restrictions of $x_n$ and $x_0$ to $[0,T]$ for $T=t_k$ for each $t_k$ in some sequence $\{t_k\}_{k\geqslant 1}$ with $t_k\to +\infty$ as $k\to+\infty$, where $\{t_k\}$ can depend on $x_0$. For the $J_1$-topology (see \cite{JS,Stone, whitt02}), it suffices to let $t_k$ be the continuity points of the limit function $x_0$. Now we will discuss the case of $M_1$-topology,  but the extension is valid to the other non-uniform topologies as well.

Let $\rho_t:D([0,+\infty),\R)\rightarrow D([0,+\infty),\R)$ be the restriction map with $\rho_t(x)(s)=x(s)$, $0\leqslant s\leqslant t$. Suppose that $g:D([0,+\infty),\R)\rightarrow D([0,+\infty),\R)$ and $g_t:D([0,t],\R)\rightarrow D([0,t],\R)$ for $t>0$ are functions with
\begin{equation}\label{eqfrest}
g_t(\rho_t(x))=\rho_t(g(x)),
\end{equation}
for all $x\in D([0,+\infty),\R)$ and all $t>0$. We then call the functions $g_t$ restrictions of the function $g$.
We have the following result about the continuity of the function $g$ in that topology:
\begin{theorem}[Theorem 12.9.1. \cite{whitt02}] Suppose that $g:D([0,+\infty),\R)\rightarrow D([0,+\infty),\R)$ has continuous restrictions $g_t$ with some topology for all $t>0$. Then $g$ itself is continuous in that topology.
\end{theorem}

We consider the extension of Lipschitz properties to subsets of $D([0,+\infty),\R)$. For this purpose, suppose that $d_{M_1,t}$ is the $M_1$-metric on  $D([0,t],\R)$ for $t>0$. An associated metric $d_{M_1,\infty}$ on $D([0,+\infty),\R)$ can be defined by 
\begin{equation}\label{intmu}
d_{M_1,\infty}(x_1,x_2)=\int_0^{+\infty} e^{-t}(d_{M_1,t}(\rho_t(x_1),\rho_t(x_2))\wedge 1)dt.
\end{equation} 
The integral of (\ref{intmu}) is well defined and can be used for the next result:
\begin{theorem}[Theorem 12.9.2. \cite{whitt02}] Let $d_{M_1,t}$ be the $M_1$-metric on $D([0,t],\R)$. For all $x_1,x_2\in D([0,+\infty),\R)$, $d_{M_1,t}(x_1,x_2)$ as a function of $t$ is right-continuous with left limits in $(0,+\infty)$ and has a right limit at $0$. Moreover, $d_{M_1,t}(x_1,x_2)$ is continuous at $t>0$ whenever $x_1$ and $x_2$ are both continuous at $t$.
\end{theorem}

Finally, we conclude this discussion with the following characterization of $M_1$-convergence in the domain $[0,+\infty)$.

\begin{theorem}[Theorem 12.9.3. \cite{whitt02}] 
Suppose that $d_{M_1,\infty}$ and $d_{M_1,t}$, $t>0$ are the $M_1$-metrics on  $D([0,+\infty),\R)$ and $D([0,t],\R)$. Then, the following are equivalent for $x$ and $x_n$, $n\geqslant 1$, in $D([0,+\infty),\R)$:
\begin{enumerate}[(i)]
\item $d_{M_1,\infty}(x_n,x)\rightarrow 0$ as $n\to +\infty$;
\item $d_{M_1,t}(\rho_t(x_n),\rho_t(x))\rightarrow 0$ as $n\to +\infty$ for all $t\notin \mathrm{Disc}(x)$, with 
\[\mathrm{Disc}(x):=\{t\in[0,T]: \; x(t-)\neq x(t)\}\] denoting the set of discontinuities of $x$;
\item there exist parametric representations $(r,u)$ and $(r_n,u_n)$ of $x$ and $x_n$ mapping $[0,+\infty)$ into the graphs such that 
\begin{equation}
\|r_n-r\|_{[0,t]}\vee \|u_n-u\|_{[0,t]}\to 0 \qquad as \; n\to +\infty,
\end{equation}
for each $t>0$.
\end{enumerate} 
\end{theorem}

The following theorem provides a useful characterization for the $M_1$-convergence in the case of stochastic processes with jumps:
\begin{theorem}[Theorem 1.6.12 \cite{silvestrov}]\label{thm: characM1}
Let  $\{X_n(t)\}_{t\geqslant 0}$ and $\{X(t)\}_{ t\geqslant 0}$ be stochastic processes. Then,   $\{X_n(t)\}_{t\geqslant 0}$ converges to $\{X(t)\}_{ t\geqslant 0}$ in the $M_1$-topology if the following two conditions are fulfilled:
\begin{enumerate}[(i)]
\item The finite-dimensional distributions of $\{X_n(t)\}_{ t\in A}$ converge in law to the corresponding ones of $\{X(t)\}_{ t\in A}$
, as $n\to+\infty$, where $A$ is a subset of $[0,+\infty)$ that is dense in this interval and contains the point $0$, and
\item the condition on $M_1$-compactness is satisfied
\begin{equation}\label{wslim}
\lim_{\delta\to 0}\limsup_{n\to +\infty} w(X_n,\delta)=0,
\end{equation} 
where the modulus of $M_1$-compactness is defined as
\begin{equation}\label{wssup1}
w(X_n,\delta):=\sup_{t\in A}w(X_n,t,\delta) ,
\end{equation}
with
\begin{equation}\label{wssup2}
w(X_n,t,\delta):=\sup_{0\vee (t-\delta)\leqslant t_1<t_2<t_3\leqslant (t+\delta) \wedge T}\{\|X_n(t_2)-[X_n(t_1),X_n(t_3)]\|\},
\end{equation}
 $[X_n(t_1),X_n(t_3)]:=\{aX_n(t_1)+(1-a)X_n(t_3):\; 0\leqslant a\leqslant 1\}$ denotes the standard segment and $\|\cdot\|$ is the norm defined in (\ref{normm1}).  Notice that the modulus of $M_1$-compactness  plays the same role for c\`{a}dl\`{a}g functions as the modulus of continuity for continuous functions 
\end{enumerate}
\end{theorem}

\section{Functional limit theorem for stochastic integrals}\label{SectionFLTSI}
This section contains the main result of this paper. That is, a functional limit theorem for the integral of 
a deterministic continuous and bounded function $f\in C_{b}(\R)$ with respect to a properly rescaled continuous time 
random walk in the space $\D$ equipped with the Skorokhod $M_1$-topology. The limit is the corresponding integral but with respect 
to a time-changed $\alpha$-stable L\'{e}vy process where the time-change is given by the functional inverse of a $\beta$-stable subordinator. 

This section is organized as follows. In Subsection \ref{introdFCLT}, we focus on providing the ingredients and preliminary results necessary to prove the functional limit theorem.
The main result of the paper comes in Subsection \ref{MTFCLT}.

\subsection{Preliminary results}\label{introdFCLT}
\noindent Let $h$ and $r$ be two positive scaling factors such that 
\begin{equation}\label{scalingf}
\lim_{h,r\to 0}\frac{h^{\alpha}}{r^{\beta}}=1,
\end{equation}
with $\alpha\in (0,2]$ and $\beta\in(0,1]$. We rescale the duration $J$ by a positive scaling factor $r$.  
The rescaled duration is denoted by $J_r:=rJ$. Similarly, we rescale the jump $Y$ by a positive real parameter $h$ getting $Y_h:=hY$.
For $\beta=1$ and $\alpha=2$, this corresponds to the typical scaling for 
Brownian motion (or the Wiener process). 

We now define the rescaled CTRW (or rescaled compound fractional Poisson process) denoted by $X_{r, h}:=\{X_{r, h}(t)\}_{t\geqslant 0}$ as follows:
\[X_{r, h}(t)=\sum_{i=1}^{N_{\beta}(t/r)}hY_i,\]
where $N_{\beta}=\{N_{\beta}(t)\}_{t\geqslant 0}$ is the fractional Poisson process.
In Subsection \ref{MTFCLT}, we will take $r=1/n$ and $h= 1/n^{\beta/\alpha}$ as scaling factors 
and the limit for $n\to+\infty$, instead of $r,h\to 0$.

For the case of weak convergence for random variables, we can mention the following result by Scalas et al. 
\cite{Sc11}. Under some suitable assumptions, for fixed $t$, the weak convergence of the random variables  
can be proved \[X_{r,h}(t)\overset{\mathcal{L}}{\Rightarrow} U_{\alpha,\beta}(t),\] when 
$h,\; r\to 0$ and where $U_{\alpha,\beta}(t)$ denotes the random variable whose distribution is characterized by 
the probability density function $u_{\alpha, \beta}(x,t)$ given in the following theorem.
\begin{theorem}[Theorem 4.2, \cite{Sc11}]
Let $X_{r, h}$ be a compound fractional Poisson process and let $h$ and $r$ be two scaling factors such that
\[{X}^{(h)}_n = h \sum_{i=1}^n Y_i,\]
\[T^{(r)}_n = r \sum_{i=1}^nJ_i,\]
and
\[\lim_{h,r\to 0}\frac{h^{\alpha}}{r^{\beta}}=1,	\]
 with $0 <\alpha\leqslant 2$ and $0 < \beta\leqslant  1$. To clarify the role of the parameter $\alpha$, further assume that,
for $h\to 0$, one has	 \[\hat{f}_Y(h\kappa)\sim 1-h^{\alpha}|\kappa|^{\alpha}\] then, for $h,\; r\to 0$ with $h^{\alpha}/r^{\beta}\to 1$, 
$f_{h{X}_{\alpha,\beta}(rt)}(x,t)$ weakly converges to 
\begin{equation}
u_{\alpha, \beta}(x,t):=\frac{1}{t^{\beta
/\alpha}}W_{\alpha,\beta}\left(\frac{x}{t^{\beta/\alpha}}\right)
\end{equation} 
with 
\begin{equation}
W_{\alpha,\beta}=\frac{1}{2\pi}\int_{-\infty}^{+\infty} d\kappa e^{-i\kappa u }E_{\beta}(-|\kappa|^{\alpha}).
\end{equation}
\end{theorem}

This weak convergence of random variables can be extended for stochastic processes but one of the inconvenients is to establish a functional central limit theorem in an appropriate functional space (see discussion in Section \ref{Skortops}). 

In the literature, we find the work by Magdziarz and Weron. In \cite{MW06}, 
they conjectured that the process $U_{\alpha,\beta}:=\{L_{\alpha}(D^{-1}_{\beta}(t))\}_{t\geqslant 0}$ (i.e. 
the symmetric $\alpha$-stable L\'{e}vy process $\{L_{\alpha}(t)\}_{t\geqslant 0}$  subordinated to the functional 
inverse $\beta$-stable subordinator $\{D^{-1}_{\beta}(t)\}_{t\geqslant 0}$) is the functional limit of the compound 
fractional Poisson process ${X}_{r,h}$. This conjecture is proved in \cite{MS04,MS10} as discussed in the following results.

For $t\geqslant 0$, we define 
\begin{equation}
T_t:=\sum_{i=1}^{\lfloor t\rfloor}J_i,
\end{equation} 
where $\lfloor t\rfloor$ denotes the integer part of $t$.
Then, taking into account the above convergence and using Example 11.2.18 of \cite{MS01}, 
the convergence of the finite-dimensional distributions of $ \{c^{-1/\beta}T_{ct}\}_{t\geqslant 0}$ to
those of  $\{D_{\beta}(t)\}_{t\geqslant 0}$ follows:
\begin{equation}\label{eqstablsubqv}
\{c^{-1/\beta}T_{ct}\}_{t\geqslant 0}\overset{\mathcal{L}}{\Rightarrow}\{D_{\beta}(t)\}_{t\geqslant 0},\qquad 
\mathrm{as} \quad c\to+\infty.
\end{equation}
We denote by $\overset{\mathcal{L}}{\Rightarrow}$ the convergence of the finite-dimensional distributions of the corresponding process.

The limiting process $\{D_{\beta}(t)\}_{t\geqslant 0}$ is called the $\beta$-stable subordinator.
In the class of strictly $\beta$-stable diffusion processes it is the positive-oriented extreme one.
It is c\`{a}dl\`{a}g, nondecreasing and with stationary increments (see for instance, \cite{Sato99}).

The functional inverse of the process $D_{\beta}(t)$ can be defined as
\begin{equation}\label{invbetsub}
D^{-1}_{\beta}(t)=\inf\{x\geqslant 0:D_{\beta}(x)>t\}.
\end{equation}
It has almost surely continuous non-decreasing sample paths and without stationary and independent increments (see \cite{MNV}).

For any integer $n\geqslant 0$ and any $t\geqslant 0$, it can be
easily proved that the number of jumps by time $t$ is at least $n$
if and only if the $n$-th jump occurs at or before $t$, therefore
the following events coincide:
\begin{equation}\label{eqTnNt}
\{T_n\leqslant t\}=\{N_{\beta}(t)\leqslant n\}.
\end{equation}

From (\ref{eqstablsubqv}) and (\ref{eqTnNt}), a
limit theorem for the properly rescaled counting process $N_{\beta}$ follows.

\begin{theorem}[Theorem 3.6, \cite{MS01}]
\begin{equation}\label{limPoiss}
\{c^{-1/\beta}N_{\beta}(ct)\}_{t\geqslant 0}\overset{\mathcal{L}}{\Rightarrow}\{D^{-1}_{\beta}(t)\}_{t\geqslant 0},\quad 
\mathrm{as} \quad c\to+\infty.
\end{equation}
\end{theorem}

As a corollary, by using Stone's Theorem \cite{Stone}, the
convergence of the stochastic processes in the Skorokhod space $D([0,\infty),\R_+)$ follows. We endow
this space with the usual $J_1$-topology introduced in Section \ref{Skortops}.

\begin{corollary}[Corollary 3.7, \cite{MS01}]\label{cor37}
\begin{equation}\label{limPoiss2}
\{c^{-1/\beta}N_{\beta}(ct)\}_{t\geqslant 0}\overset{J_1-top}{\Rightarrow}\{D^{-1}_{\beta}(t)\}_{t\geqslant 0}, \quad 
\mathrm{as} \quad c\to+\infty.
\end{equation}
\end{corollary}

Assuming that the jumps $Y_i$ belong to the strict generalized
domain of attraction of some stable law with exponent $\alpha\in
(0,2)$, then there exist $a_n>0$ such that
\[a_n\sum_{i=1}^nY_i\overset{\mathcal{L}}{\Rightarrow} \widetilde{L}_{\alpha}, \quad \mathrm{as} \quad c\to+\infty,\]
where $\widetilde{L}_{\alpha}$ denotes a random variable with symmetric $\alpha$-stable distribution.

For sake of simplicity we consider symmetric $\alpha$-stable distribution, but the results are also true without the assumption of symmetry.
 
From Example 11.2.18 \cite{MS01}, we can extend this
convergence to the convergence of the corresponding
finite-dimensional distributions for the stochastic process
$\sum_{i=1}^{[t]}Y_i$.

\begin{theorem}[ \cite{MS01}]
\begin{equation}\label{eqstabjumps}
\left \{ c^{-1/\alpha}\sum_{i=1}^{[ct]}Y_i \right \}_{t\geqslant 0}\overset{\mathcal{L}}
{\Rightarrow}\{L_{\alpha}(t)\}_{t\geqslant 0}, \quad \mathrm{when} \quad c\to +\infty.
\end{equation}
\end{theorem}

The limiting process is a symmetric $\alpha$-stable L\'{e}vy process
with sample path belonging to $D([0,+\infty),\R)$. As a corollary,  the authors of \cite{MS04} proved the convergence in the 
Skorokhod space $D([0,+\infty),\R)$ endowed with the $J_1$-topology:

\begin{corollary}[Theorem 4.1, \cite{MS04}] \label{cor:3.2}
\begin{equation}\label{eqstabjumps2}
\left \{c^{-1/\alpha}\sum_{i=1}^{[ct]}Y_i \right \}_{t\geqslant 0}\overset{J_1-top}{\Rightarrow}
\{L_{\alpha}(t)\}_{t\geqslant 0}, \quad \mathrm{when} \quad c\to +\infty.
\end{equation}
\end{corollary}

Using the results given above, one can state the following Functional Central Limit Theorem (FCLT) proved by Meerschaert and Scheffler in 
\cite{MS04} which identifies the limit process as a time-changed $\alpha$-stable L\'{e}vy process with respect to the functional inverse 
of a $\beta$-stable subordinator $\{D_{\beta}^{-1}(t)\}_{t\geqslant 0}$, with $\alpha \in (0,2]$ and $\beta\in (0,1)$.

\begin{theorem}[Theorem 4.2, \cite{MS04}] Under the
distributional assumptions considered above for the waiting times
$J_i$ and the jumps $Y_i$, we have
\begin{equation}\label{limitcomp}
\left\{c^{-\beta/\alpha}\sum_{i=1}^{N_{\beta}(t)}Y_i\right\}_{t\geqslant 0}\overset{M_1-top}{\Rightarrow}
\{L_{\alpha}(D_{\beta}^{-1}(t))\}_{t\geqslant 0}, \quad \mathrm{when} \quad c\to +\infty,
\end{equation}
in the Skorokhod space $D([0,+\infty),\R)$ endowed with the $M_1$-topology.
\end{theorem}

The functional limit of the compound Poisson process is an $\alpha$-stable L\'evy process, whereas
in the case of the compound fractional Poisson process, one gets an $\alpha$-stable L\'evy process
subordinated to the inverse $\beta$-stable subordinator.

The strategy to prove this result is the following. First, they use Corollary \ref{cor37} and Corollary \ref{cor:3.2}. 
In that case we can only ensure the convergence in $M_1$-topology instead of the convergence in $J_1$-topology because the functional 
inverse of a $\beta$-stable subordinator  $\{D_{\beta}^{-1}(t)\}_{t\geqslant 0}$ is not strictly increasing and Theorem 3.1. of \cite{whitt80} 
can not be applied. Then, they apply the Continuous Mapping Theorem together with Theorem 13.2.4. of \cite{whitt02} that provides a more 
general continuity result in $M_1$-topology and it is useful for proving the convergence in $M_1$-topology.  

By using Corollary 10.12 of \cite{Jacod}, we see that the limit process obtained in (\ref{limitcomp}), 
$X_{\alpha,\beta} (T) :=\{L_{\alpha}(D^{-1}_{\beta}(t))\}_{t\geqslant 0}$, is a semimartingale defined on the probability space 
$(\Omega,\mathcal{F} ,\mathbb{P})$. In fact, $t\rightarrow D^{-1}_{\beta}(t)$ is a non decreasing right-continuous process with 
left limits such that for each fixed $t$, the random variable $D^{-1}_{\beta}(t)$ is a stopping time with respect to the 
standard filtration $\{\F_t\}_{t\geqslant 0}$. Furthermore, $D^{-1}_{\beta}(t)$ is finite $\mathbb{P}$-a.s. for all $t\geqslant 0$ and 
that $D^{-1}_{\beta}(t) \to \infty$ as $t\to \infty$. Then, the process $\{D^{-1}_{\beta}(t)\}_{t\geqslant 0}$ defines a 
finite $\F_t$-time change. Then, it 
means that $X_{\alpha,\beta}$ is a time-changed L\'{e}vy process $\{L_{\alpha}(t)\}_{t\geqslant 0}$ with respect to the functional inverse 
of a $\beta$-stable subordinator, $\{D^{-1}_{\beta}(t)\}_{t\geqslant 0}$.

\begin{lemma}[Corollary 10.12, \cite{Jacod}]
Suppose that $T_t$ is a finite time-change with respect to the standard filtration $\{\F_t\}_{t\geqslant 0}$. 
Then, if $X$ is a semimartingale with respect to $\{\F_t\}_{t\geqslant 0}$, then  $X\circ T$ is a semimartingale 
with respect to the filtration $\{{\mathcal G}_t\}_{t\geqslant 0}$ where ${\mathcal G}_t\DEF {\mathcal F}_{T_t}$.  
\end{lemma}

Taking into account that $X_{\alpha,\beta} (t)$ is a semimartingale, it makes sense to consider a stochastic integral driven by that
process. The interested reader can also consult \cite{kobayashi2010}.

\subsection{Main theorem (case $\alpha\in(1,2]$)}\label{MTFCLT}

Our strategy to prove the convergence of the following sequence of stochastic integrals as a process in the Skorokhod space endowed with 
the $M_1$-topology and for the case of finite mean, that is $\alpha\in(1,2]$, can be stated as follows:
\begin{enumerate}[(I)]
\item We first check the $M_1$-compactness condition for the integral process $\{I_n(t)\}_{t\geqslant 0}$ defined in (\ref{defIn}).  
This result will be established in Lemma \ref{Lemmamodcont}.
\item Then, we prove the convergence in law for the family of processes $\{I_n(t)\}_{t\geqslant 0}$ when $n\to+\infty$.  
Indeed, we show that $\{X^{(n)}(t)\}_{t\geqslant 0}$ defined in (\ref{Xnn}) is uniformly tight or a \emph{good sequence} 
(see Definition \ref{good} and Lemma \ref{lemmaUT}) and, then, we can apply a combination of Theorems 7.2. and 7.4 by Kurtz and Protter 
given in \cite{KP96II} (see Proposition \ref{propconvlaw1}). 
Finally, applying the Continuous Mapping Theorem and taking the composition function as a continuous mapping, we obtain the desired convergence in law.
\item Taking into account the results proved in the previous steps, we apply Theorem \ref{thm: characM1} of Section \ref{Skortops} 
that provides a useful characterization for the $M_1$-convergence. Eventually, we obtain the main result of the paper stated as Theorem \ref{thm:MainTHM}.
\end{enumerate}

As stated in (I), in the following lemma, we can see that the second hypothesis of Theorem \ref{thm: characM1} is fulfilled.

\begin{lemma}\label{Lemmamodcont}
Let $f\in \mathcal{C}_b(\R)$ be a continuous bounded function on $\R$. Let $\{Y_i\}_{i\in\N}$ be i.i.d. symmetric 
random variables. 
Assume that $Y_1$ belongs to the domain of attraction (DOA) of an $\alpha$-stable random variable $S_{\alpha}$, with $\alpha\in (0,2]$. Let  $\{J_i\}_{i\in \N}$ be i.i.d. and positive such that $J_1$ belongs to the strict DOA of some stable random variables with 
index $\beta\in (0,1)$ and $T_n=\sum_{i=1}^n J_i$. 
Consider
\begin{equation}\label{defIn}
I_n(t):=\sum_{k=1}^{N_{\beta}(nt)}f\left(\frac{T_k}{n}\right)\frac{Y_k}{n^{\beta/\alpha}}.
\end{equation}
If \begin{equation}\label{condmodXn}
\lim_{\delta\to 0}\limsup_{n\to +\infty} w({X}_n,\delta)=0,
\end{equation}
where 
\begin{align*}
{X}_n(t):=\sum_{k=1}^{N_{\beta}(nt)}\frac{Y_k}{n^{\beta/\alpha}}=\sum_{k\geqslant 1}\frac{Y_k}{n^{\beta/\alpha}}\textbf{1}_{\{\tau_k\leqslant t\}},
\end{align*}
with $\tau_k=\inf \{t:N_{\beta}(nt)\geqslant k\}$.
Then, 
\begin{equation}\label{condmodtildeXn}
\lim_{\delta\to 0}\limsup_{n\to +\infty} w(I_n,\delta)=0.
\end{equation}
\end{lemma}
\begin{proof}
We can rewrite $I_n(t)$ as follows
\begin{align*}
I_n(t)=\sum_{k\geqslant 1}f\left(\frac{T_k}{n}\right)\frac{Y_k}{n^{\beta/\alpha}}\textbf{1}_{\{\tau_k\leqslant t\}},
\end{align*}
with $\tau_k=\inf \{t:N_{\beta}(nt)\geqslant k\}=\inf \{t:\frac{T_k}{n}\leqslant t\}$.

\noindent We want to compute $w(I_n,\delta)$. Recall that, by definition (see (\ref{wssup1}) and (\ref{wssup2})), 
\begin{equation}
w(I_n,\delta)=\sup_{t\in S}w(I_n,t,\delta),
\end{equation}
with $S\subset [0,+\infty)$ and 
\begin{equation}\label{wssup2}
w(I_n,t,\delta):=\sup_{0\vee (t-\delta)\leqslant t_1<t_2<t_3\leqslant (t+\delta) \wedge T}\{\|I_n(t_2)-[I_n(t_1),I_n(t_3)]\|\},
\end{equation}
where $\|\cdot\|$ denotes the norm defined in (\ref{normm1}).
In other words, $\|I_n(t_2)-[I_n(t_1),I_n(t_3)]\|$ denotes the distance from $I_{n}(t_2)$ to the segment $[I_n(t_1),I_n(t_3)]$; this is
given by
\[\|I_n(t_2)-[I_n(t_1),I_n(t_3)]\|=\left\{\begin{array}{cc}
0 & \mbox{if $I_n(t_2)\in [I_n(t_1),I_n(t_3)]$},\\
|I_n(t_1)-I_{n}(t_2)|\wedge|I_n(t_3)-I_{n}(t_2)| &\mbox{if $I_n(t_2)\notin [I_n(t_1),I_n(t_3)].$}
\end{array}\right.\]

\noindent Now, taking into account that the function $f$ is bounded and for $i<j$,
\begin{align*}
|I_n(t_j)-I_n(t_i)|&=\left|\sum_{k\geqslant 1}f\left(\frac{T_k}{n}\right)\frac{Y_k}{n^{\beta/\alpha}}\textbf{1}_{\{t_i<\tau_k\leqslant t_j\}}\right|
\\&\leqslant C_f\left|\sum_{k\geqslant 1}\frac{Y_k}{n^{\beta/\alpha}}\textbf{1}_{\{t_i<\tau_k\leqslant t_j\}}\right|
\\&=C_f\left|X_n(t_i)-X_n(t_j)\right|,
\end{align*}
where $C_f>0$ is a positive constant depending on the bounds of the function $f$ (remember that the sum has a finite number of terms).

\noindent For $0\vee (t-\delta)\leqslant t_1<t_2<t_3\leqslant (t+\delta) \wedge T$, we have that 
$$
0 \leqslant \|I_n(t_2)-[I_n(t_1),I_n(t_3)]\|\leqslant C_f\|{X}_n(t_2)-[{X}_n(t_1),{X}_n(t_3)]\|.
$$
This means that
\begin{equation}
0 \leqslant w(I_n,\delta)\leqslant C_f\;  w(X_n,\delta),
\end{equation}
and taking the limits on both sides
\begin{equation}
0 \leqslant \lim_{\delta\to 0}\limsup_{n\to +\infty} w(I_n,\delta)\leqslant C_f\;  \lim_{\delta\to 0}\limsup_{n\to +\infty} w(X_n,\delta)=0.
\end{equation}
Thus,  it follows that
\begin{equation*}
\lim_{\delta\to 0}\limsup_{n\to +\infty} w(I_n,\delta)=0.
\end{equation*}
\end{proof}

\vspace{1cm}

\noindent Now, to see the convergence in the $M_1$-topology it only remains to prove the following convergence of the finite-dimensional distributions:
\begin{align*}
\sum_{k=1}^{N_{\beta}(nt)}f\left(\frac{T_k}{n}\right)\frac{Y_k}{n^{\beta/\alpha}}\overset{\mathcal{L}}{\Rightarrow} \int_0^t f(s)dL_{\alpha}(D_{\beta}^{-1}(s)), \qquad n\to +\infty.
\end{align*}

A fundamental question is to know under what conditions the convergence in law of $(H^n,X^n)$ to $(H,X)$ implies that $X$ is a semimartingale and that $\int_0^tH^n(s-)dX_s^n
$ converges in law to $\int_0^tH(s-)dX_s$. In this framework, we introduce the concept of \textsl{good sequence}.

\begin{definition}[p.2, \cite{KP96I}]\label{good}
Let $(X^n)_{n\in \N}$ be an $\R^k$-valued process defined on probability space $(\Omega^n, \F^n, \mathbb{P}^n)$ 
such that it is $\F_t^n$-semimartingale. Let the sequence $(X^n)_{n\in \N}$ converge in distribution in the Skorokhod topology 
to a process $X$. The sequence $(X^n)_{n\in \N}$ is said to be good if for any sequence $(H^n)_{n\in \N}$ of $\M^{km}$-valued 
(which denotes the real-valued $k\times m$ matrices) c\`{a}dl\`{a}g processes, $H^n$ $\F_t^n$-adapted, such that $(H^n,X^n)$ converges in distribution in the Skorokhod topology on $D_{\M^{km}\times\R^m}([0,\infty))$ to a process $(H,X)$, there exists a filtration $\F_t$ such that $H$ is $\F_t$-adapted, $X$ is an $\F_t$-semimartingale, and 
\[\int_0^tH^n(s-)dX_s^n\overset{\mathcal{L}}{\Rightarrow}\int_0^tH(s-)dX_s,\]
when $n\to \infty$.
\end{definition}
Jakubowski, M\'{e}min and Pag\`{e}s \cite{JMP89} give a sufficient condition for a sequence $(X^n)_{n\in \N}$ to be good called 
\textsl{uniform tightness} (UT). This condition uses the characterization of a semimartingale as a good integrator 
(see Protter \cite{Protter}), and it holds uniformly in $n$.

\noindent In \cite{JMP89} and \cite{KP96I}, the following lemma was proved.
\begin{lemma}\label{UT1}
If $(X^n)_{n \in \N}$ is a sequence of local martingales and the following condition
\begin{equation}\label{condUT1}
\sup_{n}\E^n\left[\sup_{s\leqslant t}|\Delta X^n (s)|\right]<+\infty
\end{equation}
holds for each $t < + \infty$, where 
\begin{equation}\label{Deltan}
\Delta X^n (s):=X^n (s)-X^n (s-)
\end{equation}
denotes the jump of $X^n$ in $s$, then the sequence is uniformly tight.
\end{lemma}

\noindent In the following lemma we show that the sequence $(X^{(n)})_{n\in \N}$ defined in (\ref{Xnn}) is uniformly tight.
In this case, we have to restrict on $\alpha\in (1,2]$ if we want to prove the UT property in terms of the definition given before because in this case $Y_i$ has finite moments. In the next section, we will see how to avoid this problem and we will extend the corresponding functional theorem to $\alpha\in(0,2]$.
\begin{lemma}\label{lemmaUT}
Assume that $(Y_i)_{i\in \N}$ be i.i.d. symmetric $\alpha$-stable random variables, with $\alpha\in (1,2]$. Let 
\begin{equation}\label{Xnn}
X^{(n)}(t):=\sum_{i=1}^{\lfloor n^{\beta} t\rfloor} \frac{Y_i}{n^{\beta/\alpha}}
\end{equation} 
be defined on the probability space $(\Omega^n, \F^n,\mathbb{P}^n)$. Then $X^{(n)}(t)$ is a $\F^n_t$-martingale 
(with respect the natural filtration of $X^{(n)}$) and
\[\sup_{n}\E^n\left[\sup_{s\leqslant t}|\Delta X^{(n)}(s)|\right]<+\infty,\] for each $t<+\infty$,
where, as before
$
\Delta X^{(n)}(s):=X^{(n)}(s)-X^{(n)}(s-)$. Moreover, the elements of the sequence are
$\F^n_t$-semimartingales and the limit is also a semimartingale. 
Therefore, the sequence \eqref{Xnn} is uniformly tight as a consequence
of Lemma \ref{UT1} and it is good.
\end{lemma}

\begin{proof}
This proof is inspired by the proof of Theorem 2 in \cite{Burr2011}.

Let $\F^n_{t}=\sigma\{X^{(n)}(s),\; s\in[0,t]\}$, the natural filtration of $X^{(n)}$.  We will check that $X^{(n)}:=\{X^{(n)}(t)\}_{n\in\N}$ is an $\F^n_{t}$-martingale.  

Observe that $X^{(n)}(t)\in L^1$ for all $t\geqslant 0$,
\[\E^n\left|X^{(n)}(t)\right|=\E^n\left|\sum_{i=1}^{\lfloor n^{\beta} t\rfloor} \frac{Y_i}{n^{\beta/\alpha}}\right|\leqslant\sum_{i=1}^{\lfloor n^{\beta} t\rfloor }\E^n\left| \frac{Y_i}{n^{\beta/\alpha}}\right| =\frac{\lfloor n^{\beta} t\rfloor}{n^{\beta}}\E^n\left| \frac{Y_1}{n^{1/\alpha}}\right|<+\infty,\]
The process $X^{(n)}$ is $\F^n_{t}$-adapted. It follows that 
\begin{align*}
\E^n[ X^{(n)}(t)|\F^n_s]&=\E^n[ X^{(n)}(t)-X^{(n)}(s)|\F^n_s]+\E^n[ X^{(n)}(s)|\F^n_s]
\\&=\E^n\left[\sum_{i=\lfloor n^{\beta} s\rfloor+1}^{\lfloor n^{\beta} t\rfloor}\frac{Y_i}{n^{\beta/\alpha}} \big|\F^n_s\right]+X^{(n)}(s)
\\&=(\lfloor n^{\beta} t\rfloor - \lfloor n^{\beta}s\rfloor)\E^n\left[\frac{Y_1}{n^{\beta/\alpha}} \right]+X^{(n)}(s)
\\&=(\lfloor t\rfloor- \lfloor s\rfloor)\E^n\left[\frac{Y_1}{n^{1/\alpha}} \right]+X^{(n)}(s)=X^{(n)}(s),
\end{align*}
because $\E^n\left[\frac{Y_1}{n^{1/\alpha}} \right]=0$. Moreover, the process $X^{(n)} (t)$ is a pure jump process, therefore it is
a semimartingale and the limit for $n \to \infty$ of the sequence is an $\alpha$-stable L\'evy process which is also a semimartingale.
We shall return on this point in Proposition \ref{propconvlaw1}.

\noindent Now, we want to verify the following condition
 \[\sup_n \E^n[\sup_{s\leqslant t}|\Delta X^{(n)}(s)|]<\infty\] for each $t<\infty$.  
 
 Fix $t < \infty$.  Observe that $X^{(n)}(s)$ is a c\`{a}dl\`{a}g step process with jumps of size $|\frac{Y_i}{n^{\beta/\alpha}}|$ at 
times $i/n^{\beta}$, $i\in \N$.  Moreover, since $X^{(n)}(t)$ has finitely many jumps by time $t$, the supremum up to time $t$ can be 
replaced by a maximum up to $\lfloor n^{\beta}t \rfloor$.  So for each $t$ we need to find a uniform bound in $n$ for 
$\E^n\Big[\max_{1\leqslant i\leqslant \lfloor n^{\beta}t \rfloor}\big|\frac{Y_i}{n^{\beta/\alpha}}\big|\Big]$. 
 
We will symmetrize the sum $S(nt):=\sum_{i=1}^{\lfloor n^{\beta}t\rfloor} Y_i$. For this reason, let $(\epsilon_i)_{i\in \N}$ be a sequence of 
i.i.d. Rademacher random variables, independent of $(Y_i)_{i\in \N}$.  The Rademacher random variables are independent uniform random variables 
taking values in the set $\{-1,1\}$.Then  $|Y_i|=|\epsilon_iY_i|$ and the corresponding products $\epsilon_iY_i$ are i.i.d. in the domain of 
attraction of an $\alpha$-stable symmetric distribution.  Let $\displaystyle \widetilde{S}(nt)=\sum_{i=1}^{\lfloor n^{\beta}t \rfloor} \epsilon_iY_i$ 
be the symmetrized sum.  By using the fact that  $|Y_i|=|\epsilon_iY_i|$ and L\'{e}vy's Inequality, we have that
\begin{align}\label{expect1}
\E^n\Bigg[\max_{1\leqslant i\leqslant \lfloor n^{\beta}t \rfloor}\bigg|\frac{Y_i}{n^{\beta/\alpha}}\bigg|\Bigg]
&=\int_0^\infty \mathbb{P}^n\Bigg\{\max_{1\leqslant i\leqslant \lfloor n^{\beta}t \rfloor}\bigg|\frac{Y_i}{n^{\beta/\alpha}}\bigg| > x \Bigg\} \text{ }dx\nonumber\\
&=\int_0^\infty \mathbb{P}^n\Bigg\{\max_{1\leqslant i\leqslant \lfloor n^{\beta}t \rfloor}\bigg|\frac{\epsilon_iY_i}{n^{\beta/\alpha}}\bigg| > x \Bigg\} \text{ }dx\nonumber\\
&\leqslant \int_0^\infty 2\mathbb{P}^n\Bigg\{\bigg|\sum_{i=1}^{\lfloor n^{\beta}t \rfloor}\frac{\epsilon_iY_i}{n^{\beta/\alpha}}\bigg| > x\Bigg\} \text{ }dx\nonumber
\\&= 2\E^n\Bigg|\frac{\widetilde{S}(n^{\beta}t)}{n^{\beta/\alpha}} \Bigg|.
\end{align}

\noindent For each $n$, we have 
\begin{equation}\label{boundexpfin}
\E^n\bigg|\frac{\widetilde{S}(n^{\beta}t)}{n^{\beta/\alpha}} \bigg| \leqslant \sum_{i=1}^{\lfloor n^{\beta}t \rfloor}
\E^n\bigg|\frac{\epsilon_iY_i}{n^{\beta/\alpha}}\bigg| = \lfloor n^{\beta}t \rfloor \E^n\bigg|\frac{Y_1}{n^{\beta/\alpha}}\bigg| < \infty.
\end{equation}
Thus,  the expectation is bounded if $n$ is small.  It remains to show that the expectation is bounded in the case of large $n$.  Notice that
\begin{equation*}
\E^n\bigg|\frac{\widetilde{S}(n^{\beta}t)}{n^{\beta/\alpha}} \bigg| =
\bigg(\frac{\lfloor n^{\beta}t \rfloor}{n^{\beta}}\bigg)^{1/\alpha}
\E^n\bigg|\frac{\widetilde{S}(n^{\beta}t)}{\lfloor n^{\beta/\alpha}t\rfloor} \bigg|.
\end{equation*}

\noindent Taking into account that $\epsilon_iY_i$ are in the domain of attraction of some $\alpha$-stable law for $\alpha \in (1,2]$, then the following first moments converge: 
\begin{equation}
\E^n\Bigg|\frac{\widetilde{S}(n^{\beta}t)}{\lfloor  n^{\beta/\alpha}t \rfloor} \Bigg| \rightarrow \E|S_{\alpha}| \text{ as } n\rightarrow \infty.
\label{eqn:conv_moments}
\end{equation}

\noindent Then applying Equation (\ref{eqn:conv_moments}), there exists some $n_0 \in \N$ such that for all $n > n_0$,
\begin{equation}\label{boundexpinf}
\E^n\bigg|\frac{\widetilde{S}(n^{\beta}t)}{n^{\beta/\alpha}} \bigg| \leqslant ct^{1/\alpha}(\E|S_\alpha| + 1) < \infty ,
\end{equation}
for $c>1$, where $S_\alpha$ has a finite first moment since $\alpha\in (1,2]$.

\noindent So this means that  the desired condition holds. In fact, as a consequence of the supremum definition, we can write the following equality
\begin{align*}
\sup_n \E^n\bigg\{\sup_{s\leqslant t}|\Delta X^{(n)}(s)|\bigg\}\notag
&= \max\bigg(\max_{1\leqslant n\leqslant n_0} \E^n\Big\{\sup_{s\leqslant t}|\Delta X^{(n)}(s)|\Big\},
\sup_{n>n_0} \E^n\Big\{\sup_{s\leqslant t}|\Delta X^{(n)}(s)|\Big\}\bigg).
\end{align*}
Now, replacing (\ref{boundexpfin}) and (\ref{boundexpinf}) into (\ref{expect1}) we obtain the condition that we want to prove
\begin{align*}
\sup_n \E^n\bigg\{\sup_{s\leqslant t}|\Delta X^{(n)}(s)|\bigg\}\notag
&\leqslant \max_{1\leqslant n\leqslant n_0}2\lfloor t \rfloor \E^n\bigg|\frac{Y_1}{n^{1/\alpha}}\bigg|
+ 2ct^{1/\alpha}(\E|S_\alpha| + 1)< \infty.
\end{align*}
\end{proof}

Now, we prove the following result on convergence in law for the case $\alpha\in (1,2]$ . 
\begin{proposition}\label{propconvlaw1}
Let $f\in \mathcal{C}_b(\R)$ be a continuous bounded function on $\R$. Under the distributional assumptions and the scaling considered above  we have that 
\[\left\{\sum_{i=1}^{\lfloor n^{\beta}t\rfloor}f\left(\frac{T_i}{n}\right)\frac{Y_i}{n^{\beta/\alpha}}\right\}_{t\geqslant 0}\overset{J_1-top}{\Rightarrow}\left\{\int_0^t f(D_{\beta}(s))dL_{\alpha}(s)\right\}_{t\geqslant 0},\]
when $n\to +\infty$.
\end{proposition}
\begin{proof}
Taking into account that $f$ is a continuous function and applying the Continuous Mapping Theorem discussed in \cite{whitt02} we have that 
\[\left\{f\left(\frac{T_i}{n}\right)\right\}_{t\geqslant 0}\overset{J_1-top}{ \Rightarrow} \left\{f(D_{\beta}(t))\right\}_{t\geqslant 0},\]
when $n\to+\infty$.

\noindent Applying  Lemma \ref{lemmaUT} we have that $(X^{(n)})_{n\in\N}$ is uniformly tight (UT).  Combining Theorem 7.2 and Theorem 7.4 in \cite{KP96II} and taking into account that
 \[\left\{\left(f\left(\frac{T_i}{n}\right),X^{(n)}(t)\right)\right\}_{t\geqslant 0}\overset{J_1-top}{ \Rightarrow} \left\{(f(D_{\beta}(t)),L_{\alpha}(t))\right\}_{t\geqslant 0},\] as $n\to+\infty$, 
we conclude that $(X^{(n)})_{n\in\N}$ is a \textsl{good sequence}, as discussed in Lemma \ref{lemmaUT}. As we have seen in Definition \ref{good}, this means that 
\begin{equation*}
\left\{\left(f\left(\frac{T_i}{n}\right),X^{(n)}(t),
\sum_{i=1}^{\lfloor n^{\beta}t\rfloor}f\left(\frac{T_i}{n}\right)\frac{Y_i}{n^{\beta/\alpha}}\right)\right\}_{t\geqslant 0} \overset{J_1-top}{\Longrightarrow} \left\{\left(f(D_{\beta}(t)),
L_{\alpha}(t) ,\int_0^t f(D_{\beta}(s))dL_{\alpha}(s)\right)\right\}_{t\geqslant 0},
\end{equation*}
as $n\rightarrow \infty$.

\noindent In particular, we have that 
\[\left\{\sum_{i=1}^{\lfloor n^{\beta}t\rfloor}f\left(\frac{T_i}{n}\right)\frac{Y_i}{n^{\beta/\alpha}}\right\}_{t\geqslant 0}\overset{J_1-top}{\Rightarrow}\left\{\int_0^t f(D_{\beta}(s))dL_{\alpha}(s)\right\}_{t\geqslant 0},\]
when $n\to +\infty$.
\end{proof}

\begin{remark}
We have considered the following sum with a deterministic number of summands \[\sum_{i=1}^{\lfloor n^{\beta}t\rfloor}\frac{Y_i}{n^{\beta/\alpha}}\] because taking into account that $\frac{\lfloor n^{\beta}\rfloor}{n^{\beta}}\to 1$, when $n\to \infty$ we can replace  $n^{\beta}$  by $\lfloor n^{\beta}\rfloor$.
The reason to use the sum given by $\sum_{i=1}^{ n^{\beta}}\frac{Y_i}{n^{\beta/\alpha}}$
comes from the fact that the characteristic function of  $\sum_{i=1}^{n^{\beta}}\frac{Y_i}{n^{\beta/\alpha}}$ is
\[\E\left[e^{ik\left(\sum_{i=1}^{n^{\beta}}\frac{Y_i}{n^{\beta/\alpha}}\right)}\right]=(e^{-|k|^{\alpha}/n^{\beta}})^{n^{\beta}}=e^{-|k|^{\alpha}}.\]
\end{remark}

\begin{proposition}\label{propstocsumconv}
Let $f\in \mathcal{C}_b(\R)$ be a continuous bounded function on $\R$. Under the distributional assumptions and the scaling considered above, the integral $\left\{\sum_{i=1}^{N_{\beta}(nt)}f\left(\frac{T_i}{n}\right)Y_i\right\}_{t\geqslant 0}$ converges in law to $\left\{\int_0^{D_{\beta}^{-1}(t)}f(D_{\beta}(s))dL_{\alpha}(s)\right\}_{t\geqslant 0}$ in the Skorokhod topology as $n\to +\infty$, where 
\[\int_0^{D_{\beta}^{-1}(t)}f(D_{\beta}(s))dL_{\alpha}(s)\overset{a.s.}{=}\int_0^{t}f(s)dL_{\alpha}(D^{-1}_{\beta}(s)).\]
\end{proposition}
\begin{proof}
As a consequence of Proposition \ref{propconvlaw1} we know that 
\[\left\{\left(\sum_{i=1}^{\lfloor n^{\beta}t \rfloor}f\left(\frac{T_i}{n}\right)\frac{Y_i}{n^{\beta/\alpha}}, 
N_{\beta}^{(n)}(t)\right)\right\}_{t\geqslant 0}\overset{J_1-top}{\Rightarrow}\left\{\left(\int_0^t f(D_{\beta}(s))dL_{\alpha}(s),D^{-1}_{\beta}(t)\right)\right\}_{t\geqslant 0},\] 
in the Skorokhod space $D([0,\infty),\R)$ when $n\to +\infty$.

\noindent Given $x\in D([0,\infty),\R)$, and $y \in D([0,\infty),\R_+)$ we consider the following map 
 \begin{align*}
\Phi_1:D([0,\infty),\R)\times D([0,\infty),\R_+)&\rightarrow D([0,\infty),\R) \\
(x,y)\qquad \qquad \qquad & \rightarrow \qquad x\circ y,
\end{align*}which is continuous by Theorem 13.2.4 of \cite{whitt02}.
Now, by applying the Continuous Mapping Theorem proved in \cite{whitt02} to the previous terms $\sum_{i=1}^{\lfloor n^{\beta}t \rfloor}f\left(\frac{T_i}{n}\right)\frac{Y_i}{n^{\beta/\alpha}}$ and $D^{-1}_{\beta}(t)$, we obtain the desired convergence in law, that is, 
\[\left\{\sum_{i=1}^{N_{\beta}(nt)}f\left(\frac{T_i}{n}\right)Y_i\right\}_{t\geqslant 0}\overset {J_1-top}{\Rightarrow}\left\{\int_0^{D_{\beta}^{-1}(t)}f(D_{\beta}(s))dL_{\alpha}(s)\right\}_{t\geqslant 0}.\]

\noindent Finally, following Theorem 3.1. in \cite{kobayashi2010}, we have the following equality with probability one:
\[\int_0^{D_{\beta}^{-1}(t)}f(D_{\beta}(s))dL_{\alpha}(s)\overset{a.s.}{=}\int_0^{t}f(s)dL_{\alpha}(D^{-1}_{\beta}(s)).\]
\end{proof}

\begin{theorem}\label{thm:MainTHM}
Let $f\in \mathcal{C}_b(\R)$ be a continuous bounded function on $\R$. Under the distributional assumptions and the scaling considered above, it follows that
\[\left\{\sum_{i=1}^{N_{\beta}(nt)}f\left(\frac{T_i}{n}\right)\frac{Y_i}{n^{\beta/\alpha}}\right\}_{t\geqslant 0}\overset{M_1-top}{\Rightarrow}\left\{\int_0^{t}f(s)dL_{\alpha}(D^{-1}_{\beta}(s))\right\}_{t\geqslant 0},\]
when $n\to+\infty$, in the Skorokhod space $D([0,+\infty),\R)$ endowed with the $M_1$-topology.
\end{theorem}
\begin{proof}
To prove this result we only need to check that the hypothesis of Theorem \ref{thm: characM1} are fullfilled. By applying Lemma \ref{Lemmamodcont} the first hypothesis is satisfied. Then, using Proposition \ref{propstocsumconv} we obtain the convergence in the $J_1$-topology and in particular, we get the second one. Thus, the convergence of 
\[\left\{\sum_{i=1}^{N_{\beta}(nt)}f\left(\frac{T_i}{n}\right)\frac{Y_i}{n^{\beta/\alpha}}\right\}_{t\geqslant 0}\overset{M_1-top}{\Rightarrow}\left\{\int_0^{t}f(s)dL_{\alpha}(D^{-1}_{\beta}(s))\right\}_{t\geqslant 0},\]
in the Skorohod space endowed with the $M_1$-topology is proved.
\end{proof}

\subsection{Generalization of the main theorem (case $\alpha\in(0,2]$)}
In this subsection we will extend the Functional Central Limit Theorem (FCLT) proved above to the case where $(Y_i)_{i\in\N}$ is a sequence of random variables in the domain of attraction of an $\alpha$-stable process with index $\alpha\in(0,2]$.  In particular, we are interested in the weak convergence in the Skorokhod $M_1$-topology of 
\begin{equation}
\left\{\sum_{i=1}^{N_{\beta}(nt)}f\left(\frac{T_i}{n}\right)\frac{Y_i}{n^{\beta/\alpha}}\right\}_{t\geqslant 0}\overset{M_1-top}{\Rightarrow}\left\{\int_0^tf(s)dL_{\alpha}(D_{\beta}^{-1}(s))\right\}_{t\geqslant 0},\; n\to\infty.
\end{equation}

The methodology used in the previous subsection to prove the FCLT for the case with index $\alpha\in (1,2]$ was based on the condition (\ref{condUT1}) of  Lemma \ref{UT1}. This condition is not useful for the case of jumps with $\alpha$-stable distribution with index $\alpha\in (0,1]$ because they have infinite moments.

To solve this problem, we will use another definition of uniformly tightness and a criterion given by Kurtz and Protter in \cite{KP91} which will be more convenient to use this result for our case.

\begin{definition}\label{defUT}

A sequence of $\F_t^n$-semimartingales $(X^n)_{n \in \N}$ is uniformly tight if for each $t > 0$, the set
$\{ \int_0^t H^n(s-) dX_s^n, \, H^n \in \mathbf{S}^n, \, |H^n| \leqslant 1, \, n \geqslant 1 \}$ is stochastically bounded, uniformly in $n$, where $(H^n)_{n\in \N}$ denotes a sequence of $\M^{km}$-valued  c\`{a}dl\`{a}g processes and $\F_t^n$-adapted and $\mathbf{S}^n$ denotes the class of simple predictable processes with the respect to the sequence of filtered probability spaces.
\end{definition}

This definition gives a theoretically criterion but it is not easy to verify in practice. In \cite{KP91} we find the following standard procedure in the theory of stochastic integration which is also a Skorokhod continuous procedure and consists on shrinking the large jumps to be no larger than a specified $\delta>0$.

We define $h_{\delta}:\R_+\rightarrow \R_{+}$ by $h_{\delta}(r):=(1-\delta/r)^+$, and $K_{\delta}:D(\R^d)\rightarrow D(\R^d)$ by 
\begin{equation*}
K_{\delta}(x)t:=\sum_{0<s\leqslant t}h_{\delta}(|\Delta x_s|)\Delta x_s.
\end{equation*}

For a semimartingale $X$, set $X^{\delta}:=X-K_{\delta}(X)$, and for a sequence $X^{(n)}$: $X^{(n),\delta}:=X^{(n)}-K_{\delta}(X^{(n)})$. Then the semimartingale $X^{\delta}$ has all of its jumps bounded by $\delta$. It is well-known that a semimartingale with bounded jumps has many good properties.

Considering the sequence of semimartingales $X^{(n),\delta}$ and $\delta>0$, for each $n$, $X^{(n),\delta}$ has a decomposition 
\begin{equation}
X^{(n),\delta}=M^{n,\delta}+A^{n,\delta},
\end{equation}
where $M^{n,\delta}$ is a local martingale and $A^{n,\delta}$ has paths of finite variation on compacts, and $M^{n,\delta}$ and $A^{n,\delta}$  both have bounded jumps.

The quadratic variation $[M^{n,\delta}, M^{n,\delta}]$ and $A^{n,\delta}$ are locally bounded and the total variation process of $A^{n,\delta}$, denoted by $TV(A^{n,\delta})$ is also locally bounded. That is there exist stopping times $\{\tau_{n}^{k}\}_{k\in\N}$ increasing to infinity a.s. in $k$  such that $[M^{n,\delta}, M^{n,\delta}]$ and $TV(A^{n,\delta})$  are bounded a.s.

The following result by Kurtz and Protter is similar to that proven by A. Jakubowski, J. M\'{e}min and G. Pag\`{e}s in \cite{JMP89} but the method of the proof is entirely different. In our case it is more convenient to use this result.

\begin{theorem}[Theorem 2.2, \cite{KP91}]\label{th22KP}
For each $n$, let $(H^n,X^n)$ be an $\F_t^n$-adapted process with sample paths in $D_{\M^{km}\times\R^m}([0,\infty))$ and let $X^n$  be an $\F_t^n$-semimartingale. Fix $\delta >0$ and define $X^{n,\delta}=X^n-K_{\delta}(X^n)$. (Note that $X^{n,\delta}$ will also be a semimartingale.) Let $X^{n,\delta}=M^{n,\delta}+A^{n,\delta}$ be a decomposition of $X^{n,\delta}$ into an $\F_t^n$-local martingale and a process with finite variation. Suppose 
\begin{itemize}
\item[(C1)] for each $\theta>0$, there exist stopping times $\tau_n^{\theta}$ such that $\mathbb{P}(\tau_n^{{\theta}}\leqslant \theta)\leqslant  \frac{1}{\theta}$ and furthermore
\begin{equation}\label{cond22iKP}
\sup_n\E\left[[M^{n,\delta},M^{n,\delta}]_{t\wedge\tau_n^{\theta}}+TV(A^{n,\delta},{t\wedge\tau_n^{\theta}})\right]<+\infty.
\end{equation}
\end{itemize} 
If $(H^n,X^n)\Rightarrow (H,X)$ in the Skorokhod topology on  $D_{\M^{km}\times\R^m}([0,\infty))$, then $X$ is a semimartingale with respect to a filtration to which $H$ and $X$ are adapted and $$\left(H^n,X^n,\int H^ndX^n\right)\Rightarrow \left(H,X,\int H dX\right),$$ in the Skorokhod topology on $D_{\M^{km}\times\R^m\times \R^k}([0,\infty))$.
\end{theorem}

In the next lemma, we will verify the condition (C1) of Theorem \ref{th22KP} for the stochastic process given by 
$$X^{(n)}(t):=\sum_{i=1}^{\lfloor n^{\beta} t\rfloor} \frac{Y_i}{n^{\beta/\alpha}},$$
where $(Y_i)_{i\in \N}$ are  i.i.d.  symmetric $\alpha$-stable random variables, with $\alpha\in (0,2]$.
\begin{lemma}\label{lemmaUT2}
Assume that $\{Y_i\}_{i\in \N}$ be i.i.d.  symmetric $\alpha$-stable random variables, with $\alpha\in (0,2]$. Let 
\begin{equation}\label{Xnn2}
X^{(n)}(t):=\sum_{i=1}^{\lfloor n^{\beta} t\rfloor} \frac{Y_i}{n^{\beta/\alpha}}
\end{equation} 
be defined on the probability space $(\Omega^n, \F^n,\mathbb{P}^n)$. Fix $\delta>0$ and define $X^{(n),\delta}=X^{(n)}-K_{\delta}(X^{(n)})$. Let $X^{(n),\delta}=M^{n,\delta}+A^{n,\delta}$ be a decomposition of $X^{(n),\delta}$ into a $\F^n_t$-martingale and a process with finite variation. 
Then $X^{(n),\delta}$ satisfies the condition (C1) of Theorem \ref{th22KP}.
\end{lemma}

\begin{proof}
This proof is inspired by the results obtained in Example 4.1 of \cite{Janicki96}.

Fix $\delta>0$, define $X^{(n),\delta}:=X^{(n)}-K_{\delta}(X^{(n)})$, where
\begin{equation*}
K_{\delta}(x):=\sum_{0<s\leqslant t}\left(1-\frac{\delta}{|x(s)-x(s-)|}\right)^+(x(s)-x(s-)).
\end{equation*}

Then the process $X^{(n),\delta}$ can be written as 
\begin{equation}
X^{(n),\delta}(t)=\sum_{k=1}^{\lfloor n^{\beta}t\rfloor}\left[\left(\frac{Y_k}{n^{\beta/\alpha}}\right)\mathbf{1}_{(-\delta,\delta)}\left(\frac{Y_k}{n^{\beta/\alpha}}\right)+\delta\,\text{sgn} \left(\frac{Y_k}{n^{\beta/\alpha}}\right)\mathbf{1}_{(-\infty,-\delta]\cup[\delta,\infty)}\left(\frac{Y_k}{n^{\beta/\alpha}}\right)\right].
\end{equation}

The total variation of the process $X^{(n),\delta}$ on the interval $[0,t]$ is defined with the following formula
\begin{equation}\label{TVproc}
TV(X^{(n),\delta},t):=\sup_{n}\sup_{0\leqslant t_0<t_1<\dots<t_n\leqslant t}\sum_{i=1}^n|X^{(n),\delta}_{t_{i}}-X^{(n),\delta}_{t_{i-1}}|
\end{equation}

We notice that in our case the total variation of the process
\begin{equation}\label{TVXndelta}
TV(X^{(n),\delta},t)=\sum_{k=1}^{\lfloor n^{\beta}t\rfloor}\left[\left|\frac{Y_k}{n^{\beta/\alpha}}\right|\mathbf{1}_{(-\delta,\delta)}\left(\frac{Y_k}{n^{\beta/\alpha}}\right)+\delta\mathbf{1}_{(-\infty,-\delta]\cup[\delta,\infty)}\left(\frac{Y_k}{n^{\beta/\alpha}}\right)\right],
\end{equation}
is finite. 

This means that we can take $A^{n,\delta}=X^{(n),\delta}$ and martingale part $M^{n,\delta}=0$.

Now, we define a stopping time $\tau_n^{d}=\inf\{t>0:TV(X^{(n),\delta},t)\geqslant d\}$.
Since $X^{(n),\delta}$ has finite variation, then $\tau_{n}^d<+\infty$. 

Then for $\theta>0$ there exists $d_{\theta}$ such that 
\begin{equation}
\mathbb{P}(\tau_n^{d_{\theta}}\leqslant \theta)=\mathbb{P}(TV(X^{(n),\delta},t)\geqslant d_{\theta})\leqslant \frac{1}{\theta}.
\end{equation} 

We have that
\begin{equation*}
TV(X^{(n),\delta},t \wedge \tau_n^{d_\theta})\leqslant d_{\theta}+2\delta
\end{equation*}

Now we get that 
\begin{equation*}
\sup_n\E[TV(X^{(n),\delta},t\wedge\tau_n^{d_\theta})]\leqslant d_{\theta}+2\delta<+\infty.
\end{equation*}

Hence, this implies the condition (C1) of Theorem \ref{th22KP}. That is, for each $\theta>0$, we have that 
\begin{align*}
\sup_n\E[[M^{n,\delta}, M^{n,\delta}]_{t\wedge\tau_n^{\theta}}+TV(A^{n,\delta},{t\wedge\tau_n^{\theta}})]=\sup_n\E[TV(X^{(n),\delta},t\wedge\tau_n^{d_\theta})]\leqslant d_{\theta}+2\delta<+\infty.\end{align*}
This concludes the proof.
\end{proof}

Now, we prove an extension of Proposition \ref{propconvlaw1} to the case of jumps in the domain of attraction of an $\alpha$-stable process with index $\alpha\in(0,2]$. 
\begin{proposition}\label{propconvlaw2}
Let $f\in \mathcal{C}_b(\R)$ be a continuous bounded function on $\R$. Let $\{Y_i\}_{i\in\N}$ be random variables in the domain of attraction of an $\alpha$-stable process with index $\alpha\in(0,2]$ and the waiting times $\{J_i\}_{i\in\N}$ be random variables in the domain of attraction of a $\beta$-stable process with index $\beta\in (0,1)$. Let $T_i=\sum_{l=1}^iJ_l$. Under these distributional assumptions and a suitable scaling we have that 
\[\left\{\sum_{i=1}^{\lfloor n^{\beta}t\rfloor}f\left(\frac{T_i}{n}\right)\frac{Y_i}{n^{\beta/\alpha}}\right\}_{t\geqslant 0}\overset{J_1-top}{\Rightarrow}\left\{\int_0^t f(D_{\beta}(s))dL_{\alpha}(s)\right\}_{t\geqslant 0},\]
when $n\to +\infty$.
\end{proposition}

\begin{proof}
Taking into account that $f$ is a continuous function and applying the Continuous Mapping Theorem discussed in \cite{whitt02} we have that 
\[\left\{f\left(\frac{T_i}{n}\right)\right\}_{t\geqslant 0}\overset{J_1-top}{ \Rightarrow} \left\{f(D_{\beta}(t))\right\}_{t\geqslant 0},\]
when $n\to+\infty$.
\newline Applying  Lemma \ref{lemmaUT2}, the condition (C1) of Theorem \ref{th22KP} is satisfied. Furthermore, since we have the convergence on the $J_1$-topology
 \[\left\{\left(f\left(\frac{T_i}{n}\right),X^{(n)}(t)\right)\right\}_{t\geqslant 0}\overset{J_1-top}{ \Rightarrow}\left\{ (f(D_{\beta}(t)),L_{\alpha}(t))\right\}_{t\geqslant 0},\] as $n\to+\infty$, we are under the assumptions of Theorem \ref{th22KP}. Thus, applying this theorem we conclude that
\begin{equation*}
\left\{\left(f\left(\frac{T_i}{n}\right),X^{(n)}(t),
\sum_{i=1}^{\lfloor n^{\beta}t\rfloor}f\left(\frac{T_i}{n}\right)\frac{Y_i}{n^{\beta/\alpha}}\right)\right\}_{t\geqslant 0} \overset{J_1-top}{\Longrightarrow} \left\{\left(f(D_{\beta}(t)),
L_{\alpha}(t) ,\int_0^t f(D_{\beta}(s))dL_{\alpha}(s)\right)\right\}_{t\geqslant 0},
\end{equation*}
as $n\rightarrow \infty$.
\newline In particular, we deduce that 
\[\left\{\sum_{i=1}^{\lfloor n^{\beta}t\rfloor}f\left(\frac{T_i}{n}\right)\frac{Y_i}{n^{\beta/\alpha}}\right\}_{t\geqslant 0}\overset{J_1-top}{\Rightarrow}\left\{\int_0^t f(D_{\beta}(s))dL_{\alpha}(s)\right\}_{t\geqslant 0},\]
when $n\to +\infty$.
\end{proof}

Next result is the natural extension of Proposition \ref{propstocsumconv} and Theorem \ref{thm:MainTHM} to the case $\alpha\in (0,2]$.
\begin{theorem}\label{propstocsumconvex2}
Let $f\in \mathcal{C}_b(\R)$ be a continuous bounded function on $\R$. Under the distributional assumptions and the scaling considered above, the integral $\left\{\sum_{i=1}^{N_{\beta}(nt)}f\left(\frac{T_i}{n}\right)Y_i\right\}_{t\geqslant 0}$ converges in law to $\left\{\int_0^{D_{\beta}^{-1}(t)}f(D_{\beta}(s))dL_{\alpha}(s)\right\}_{t\geqslant 0}$in the Skorokhod topology as $n\to +\infty$, where 
\[\int_0^{D_{\beta}^{-1}(t)}f(D_{\beta}(s))dL_{\alpha}(s)\overset{a.s.}{=}\int_0^{t}f(s)dL_{\alpha}(D^{-1}_{\beta}(s)).\]
Moreover, we have the convergence in the Skorokhod space $D([0,+\infty),\R)$ endowed with the $M_1$-topology
\begin{equation}
\left\{\sum_{i=1}^{N_{\beta}(nt)}f\left(\frac{T_i}{n}\right)\frac{Y_i}{n^{\beta/\alpha}}\right\}_{t\geqslant 0}\overset{M_1-top}{\Rightarrow}\left\{\int_0^t f(s)dL_{\alpha}(D^{-1}_{\beta}(s))\right\}_{t\geqslant 0},\quad n\to+\infty.
\end{equation}
\end{theorem}

\begin{proof}
From Proposition \ref{propconvlaw2} and applying Continuous Mapping Theorem on $D([0,+\infty),\R^d\times \R_+)$ taking $\Phi_2(x,y):=(x,y^{-1})$  as the continuous map (see Theorem 7.1. by \cite{whitt80}) we have that
\begin{equation*}
\left\{\left(\sum_{k=1}^{[n^\beta t]}f\left(\frac{T_k}{n}\right)\frac{Y_k}{n^{\beta/\alpha}}, n^{-1/\beta}N_{\beta}(nt)\right) \right\}_{t\geqslant 0}\overset{J_1-top}{\Rightarrow}\left\{\left(\int_0^tf(D_{\beta}(s))dL_{\alpha}(s), D^{-1}_{\beta}(t)\right) \right\}_{t\geqslant 0},
\end{equation*}
as $n\to\infty$, where $\{D^{-1}_{\beta}(t)\}_{t\geqslant 0}$ is the functional inverse of the subordinator $\{D_{\beta}(t)\}_{t\geqslant 0}$.
\newline We consider the composition map as $\Phi_1(x,y):=x\circ y$ which is continuous by Theorem 13.2.4 of \cite{whitt02}. Then, applying Continuous Mapping Theorem and Theorem 3.4.4. of \cite{whitt02}, we obtain the desired result
\begin{equation*}
\left\{\sum_{k=1}^{N_{\beta}(nt)}f\left(\frac{T_k}{n}\right)\frac{Y_k}{n^{\beta/\alpha}}\right\}_{t\geqslant 0} \overset{M_1-top}{\Rightarrow}\left\{\int_0^{D_{\beta}^{-1}(t)}f(D_{\beta}(s))dL_{\alpha}(s) \right\}_{t\geqslant 0},
\end{equation*}
as $n\to\infty$.
\noindent Now using the equality with probability one:
\[\int_0^{D_{\beta}^{-1}(t)}f(D_{\beta}(s))dL_{\alpha}(s)\overset{a.s.}{=}\int_0^{t}f(s)dL_{\alpha}(D^{-1}_{\beta}(s)),\]
we conclude that 
\begin{equation*}
\left\{\sum_{k=1}^{N_{\beta}(nt)}f\left(\frac{T_k}{n}\right)\frac{Y_k}{n^{\beta/\alpha}}\right\}_{t\geqslant 0} \overset{M_1-top}{\Rightarrow}\left\{\int_0^{t}f(s)dL_{\alpha}(D^{-1}_{\beta}(s)) \right\}_{t\geqslant 0},
\end{equation*}
when $n\to \infty$.
\end{proof}
 
As a remark, we think that it is possible to extend this considerations to more general subordinators by approximating these subordinators by compound Poisson processes but since we are interested in the stochastic differential equation given in (\ref{stringeq}), it makes sense to consider as a noise term, a formal derivative of an $\alpha$-stable L\'{e}vy process.

Finally, to conclude our work we will see the following result which is a corollary of Theorem \ref{propstocsumconvex2}.
\begin{corollary}
Let $G(t)$ be the Green function of equation (\ref{stringeqsolx}) defined in (\ref{Gt}) and $G_v(t)$ its derivative. 
Let $\{Y_i\}_{i\in\N}$ be i.i.d. symmetric $\alpha$-stable random variables that represent the particle jumps. Assume 
$Y_1$ belongs to the DOA of an $\alpha$-stable random variable $S_{\alpha}$, with $\alpha\in (0,2]$. Let  $\{J_i\}_{i\in \N}$ be 
the waiting times preceding the corresponding jumps. Assume that $\{J_i\}_{i\in \N}$ are i.i.d. such that $J_1$ belongs to the strict 
DOA of some stable random variables with index $\beta\in (0,1)$ and $T_n=\sum_{i=1}^n J_i$.  Then, under the scaling considered in Section \ref{SectionFLTSI},  it follows that
\begin{equation}\label{vcr}
v^n:=\left\{\sum_{i=1}^{N_{\beta}(nt)}G_v\left(t-\frac{T_i}{n}\right)\frac{Y_i}{n^{\beta/\alpha}}\right\}_{t\geqslant 0}\overset{M_1-top}{\Rightarrow}\left\{\int_0^{t}G_v(t-s)dL_{\alpha}(D^{-1}_{\beta}(s))\right\}_{t\geqslant 0}=:v,
\end{equation}
and
\begin{equation}\label{xcr}
x^n:=\left\{\sum_{i=1}^{N_{\beta}(nt)}G \left(t-\frac{T_i}{n}\right)\frac{Y_i}{n^{\beta/\alpha}}\right\}_{t\geqslant 0}\overset{M_1-top}{\Rightarrow}\left\{\int_0^{t}G (t-s)dL_{\alpha}(D^{-1}_{\beta}(s))\right\}_{t\geqslant 0}=:x,
\end{equation}
when $n\to+\infty$, in the Skorokhod space $D([0,+\infty),\R)$ endowed with the $M_1$-topology.
\end{corollary}

\begin{proof}
Using the convergence result stated at (\ref{eqstablsubqv}) and the fact that $t$ is a deterministic function independent of $D_{\beta}$, we have that 
\begin{equation*}
\left\{\left(t,\frac{T_i}{n}\right)\right\}_{t\geqslant 0}\overset{J_1-top}{\Rightarrow} \left\{(t,D_{\beta})\right\}_{t\geqslant 0},
\end{equation*}
in the Skorokhod topology, when $n\to\infty$.
If we apply Continuous Mapping Theorem considering the continuous function (see Theorem 4.1. of \cite{whitt80}) defined given $x\in D([0,\infty),\R)$, and $y \in D([0,\infty),\R_+)$  as
 \begin{align*}
\Phi:D([0,\infty),\R)\times D([0,\infty),\R_+)&\rightarrow D([0,\infty),\R) \\
(x,y)\qquad \qquad \qquad & \rightarrow \qquad x-y,
\end{align*}
we have that 
\begin{equation*}
\left\{t-\frac{T_i}{n}\right\}_{t\geqslant 0}\overset{J_1-top}{\Rightarrow} \left\{t-D_{\beta}(t)\right\}_{t\geqslant 0}
\end{equation*}
Now, applying again Continuous Mapping Theorem but taking as a continuous map the Green function $G$, we obtain that 
\begin{equation*}
\left\{G\left(t-\frac{T_i}{n}\right)\right\}_{t\geqslant 0}\overset{J_1-top}{\Rightarrow} \left\{G\left(t-D_{\beta}(t)\right)\right\}_{t\geqslant 0}
\end{equation*}
We know that the Green function is also a bounded function. Hence, following the steps of the proof of Theorem \ref{thm:MainTHM}, we obtain the $M_1$-convergence results presented in (\ref{xcr}) and (\ref{vcr}).
 \end{proof}

\section{Summary and outlook}

In this paper, motivated by the stochastic differential equation \eqref{stringeq}, we have studied the convergence of a class of stochastic integrals 
with respect to the compound fractional Poisson process. It turns out that, under proper scaling hypotheses, these integrals converge to integrals 
with respect to a symmetric $\alpha$-stable process subordinated to the inverse $\beta$-stable subordinator. It is therefore possible to approximate 
some of the integrals discussed in \cite{kobayashi2010} by means of simple Monte Carlo simulations. This will be the subject of a forthcoming 
applied paper. Here, we have focused on stochastic integrals of 
bounded deterministic functions. The extension of this result to the integration of stochastic processes will be a natural development. It will be also of interest to extend the convergence of stochastic integrals with respect to more generic L\'{e}vy processes and more general subordinators.
Another possible extension would be studying the coupled case (the jumps and the waiting times are not necessarily independents) using the techniques described in the papers by \cite{BK2013} and \cite{SH2011}.

\bigskip
\noindent \textbf{Acknowledgment}

{E. S. acknowledges financial support from the Italian grant PRIN 2009 `Finitary and non-finitary probabilistic methods in economics' 2009H8WPX5\_002.
N. V. acknowledges  the partial support from by the Spanish grant MEC-FEDER Ref. MTM2009-08869 from the Direcci\'{o}n General de Investigaci\'{o}n, MEC (Spain).}




\end{document}